\documentclass[12pt]{article}
\usepackage[utf8]{inputenc}
\usepackage{amsmath, amssymb, amsthm, mathtools}
\usepackage{bbold,tabularx}
\usepackage{enumerate,float,subcaption,xcolor,multirow}
\usepackage{tikz, tikzpagenodes}
\usetikzlibrary{through,intersections}
\usepackage{tikz-cd}
\mathtoolsset{centercolon}
\usepackage{xurl}

\usepackage[T1]{fontenc}
\usepackage{verbatim}
\usepackage[margin=3cm]{geometry}
\usepackage[algo2e,ruled,vlined]{algorithm2e}
\usepackage{algorithm}
\usepackage{here}
\usepackage{setspace}
\usepackage[normalem]{ulem}
\usepackage{algpseudocode}
\usepackage{dsfont}
\usepackage{url}
\usepackage[inline]{enumitem}
\usepackage{microtype}
\usepackage[affil-sl]{authblk}
\usepackage{graphicx}
\usepackage{nicematrix}
\usepackage{nicefrac}
\usepackage{amsfonts,latexsym,graphics}
\usepackage{color}
\usepackage[hidelinks]{hyperref}
\usepackage{adjustbox}
\usepackage{booktabs}
\usepackage{makeidx}
\usepackage{bm}

\usepackage{tikz}
\usetikzlibrary{tikzmark}
\usetikzlibrary{positioning}
\usetikzlibrary{graphs}
\usetikzlibrary{graphs.standard}
\usetikzlibrary{decorations.pathmorphing}
\usetikzlibrary{decorations.text}
\usetikzlibrary{arrows.meta}
\pgfdeclarelayer{bg}    % declare background layer
\pgfsetlayers{bg,main}  % set the order of the layers (main is the standard layer)

\usepackage{xspace}         % must be loaded last

\newcommand{\norm}[2]{\| #1 \|_{#2}}
\newcommand{\define}{\coloneqq}

\makeatletter
\newcommand{\rref}[2]{\hyperref[#2]{{#1}~\ref*{#2}}}

\usepackage{amsfonts,amsthm,amsmath,mathtools,amssymb,comment,xcolor,enumitem}
\usepackage{mathtools}
\usepackage{soul}

\theoremstyle{plain}
\newtheorem{thm}{Theorem}[section]

\newtheorem{lem}{Lemma}[section]
\newtheorem{cor}{Corollary}[section]

\theoremstyle{definition}
\newtheorem{df}{Definition}[section]
\newtheorem{rem}{Remark}[section]
\newtheorem{ex}{Example}[section]

\newtheorem{propo}{Proposition}[section]

\newcommand{\allone}{\mathbf{1}}

\newcommand\tran{\mkern-2mu\raise1.25ex\hbox{$\scriptscriptstyle\top$}\mkern-3.5mu}

\def\G{\Gamma}

\def\diag{\mathop{\rm diag }\nolimits}

\def\matrix0{{\mbox {\boldmath $O$}}}

\def\j{{\mbox{\boldmath $1$}}}
\def\vec0{\mbox{\bf 0}}

\newcommand{\revA}[2]{{\color{black}#2}}
\newcommand{\revC}[2]{{\color{black}#2}}

\title{An \textcolor{black}{algebraic-combinatorial} framework for finding the average hitting times in graphs with high regularity}

\author{Aida Abiad\thanks{\texttt{a.abiad.monge@tue.nl}, Department of Mathematics and Computer Science, Eindhoven University of Technology, The Netherlands} \thanks{Department of Mathematics and Data Science of Vrije Universiteit Brussel, Belgium} \qquad \qquad Yusaku Nishimura\thanks{\texttt{n2357y@ruri.waseda.jp}, School of Fundamental Science and Engineering, Waseda University, Tokyo, Japan} }

\date{}

\begin{document}

\maketitle

\begin{abstract}
For any given vertices $u$ and $v$ in a graph, the hitting time of a random walk on a finite graph is the number of steps it takes for a random walk to reach vertex $v$ starting at vertex $u$. The expected value of the hitting time is the average hitting time. In this paper, we present an \textcolor{black}{algebraic-combinatorial} method for calculating the average hitting time between vertices of finite graphs exhibiting high regularity, along with its applications to multiple graph classes. \textcolor{black}{Our approach exploits a novel connection between maximal-entropy random walks and weight-equitable partitions, providing a unifying framework that strengthens and extends several known results, including Rao’s method [\emph{Statistics \& Probability Letters}, 2013] for computing the hitting time from a vertex to a neighbor under certain symmetries of the starting vertex.} \\

\end{abstract}

%Submitted to (2025-10-08): https://amc-journal.eu/index.php/amc/login 

%%%%%%%%%%%%%%%%%%%%%%%%%%%%%%%%%%%%%%%%%%%%%%%%%%%%%%%%%%%%%%%%%%%%%%%%%%%%%%%%%%%%%%%%%%
\section{Introduction}
%%%%%%%%%%%%%%%%%%%%%%%%%%%%%%%%%%%%%%%%%%%%%%%%%%%%%%%%%%%%%%%%%%%%%%%%%%%%%%%%%%%%%%%%%%

A \emph{random walk} on a graph is a stochastic process that moves randomly from vertex to vertex along the edges. Specifically, when each vertex
is chosen with equal probability during a random walk, it is referred to as a \emph{simple random walk}. If a graph $G$ has weighted vertices, then the transition probability $t_{vu}$ from vertex $u$ to vertex $v$ of a simple random walk in $G$ is defined as 
  \[
    t_{vu}=\frac{w(v)}{\sum_{x\in G(u)} w(x)},
  \]
  where $w(v)$ is \revC{}{the} weight of $v$ and $G(u)$ are the neighbors of $u$.
For any given vertices $u$ and $v$ in a graph,
the \emph{hitting time} from vertex $u$ to vertex $v$ is defined as the number of steps required to reach vertex $v$ for the first time in a simple random
walk starting from vertex $u$. The expected value of the hitting time from vertex $u$ to vertex $v$ is called the \emph{average hitting time} from $u$ to $v$, and it is also referred to as the \emph{mean first passage time} from $u$ to $v$. We denote it as $H(G;(v,u))$, and we define $H(G;(v,v)) = 0$. We should note that a random walk that is weighted based on the eigenvectors of the adjacency matrix of a graph is denoted in the literature as the \emph{maximal entropy random walk}, \revC{see \cite{BDLW2009,DL2011,D2012,NLWY2018,OB2013,SGLNL2011}.}{, and it has been studied, among others, by Burda, Duda, Luck, and Waclaw~\cite{BDLW2009}, Devenne and Libert~\cite{DL2011}, Duda~\cite{D2012}, Niu, Liu, Wang and Yan~\cite{NLWY2018}, Ochab and Burda~\cite{OB2013}, and Sinatra, G\'omez-Garde\~nes, Lambiotte, Nicosia, and Latora~\cite{SGLNL2011}. }

%The average hitting time of a random walk has many useful properties. For example, \revC{the cover time, the expected time it takes for a random walk to visit all vertices of a graph, can be both bounded above by a function of the largest hitting time from one vertex to another, and below by a function of the smallest hitting time from one vertex to another \cite{L1993}. {Lov\'asz \cite{L1993} showed that the cover time (that is, the expected time it takes for a random walk to visit all vertices of a graph), can be both bounded above by a function of the largest hitting time from one vertex to another, and below by a function of the smallest hitting time from one vertex to another.} However, the average hitting times between vertices of various graphs can be hard to analyze, and finding their values is not intuitive. Despite this, the average hitting time between any two vertices has been found for several graph classes. 

\textcolor{black}{The average hitting time of a random walk exhibits several useful properties. For instance, Lovász \cite{L1993} showed that the cover time (the expected time for a random walk to visit all vertices of a graph) can be bounded both above, in terms of the largest hitting time between any pair of vertices, and below, in terms of the smallest such hitting time. Nevertheless, computing the average hitting times between vertices in general graphs remains a challenging problem. Despite this, exact expressions for the average hitting time have been determined for several specific classes of graphs.} Regarding graph classes with high regularity, potential theory methods have been used to derive average hitting times for distance-regular graphs by Biggs \cite{B1993} (and independently by Devroye and Sbihi \cite{DS1990}, and Van Slijpe \cite{S1984}) and for distance-biregular graphs by \revC{Carmona et al.}{Carmona, Encinas and Jiménez}~\cite{cej2022}. Nishimura \cite{N2023} introduced $f$-equitable graphs, which are yet another generalization of distance-regular graphs, and showed that their average hitting times can be calculated using the entries of a certain quotient matrix.
Rao \cite{R2013} showed a method \revC{which}{that} finds the average hitting time from one vertex to a neighbor when the graph exhibits certain symmetry on the starting vertex, and applied this method to random walks on a variety of classes of graphs\revC{}{,} including grids, hypercubes, and trees.
Recently, Tanaka~\cite{T2024, T2024-2} derived explicit formulas for the average hitting times of certain Cayley graphs.

\textcolor{black}{
In this paper, we introduce a general algebraic-combinatorial framework for deriving the average hitting time of various graph classes satisfying certain regularity conditions, thereby unifying and extending several results previously appearing in the literature. In particular, our approach includes and generalizes known results for distance-regular and distance-biregular graphs (such as Van Slijpe~\cite{S1984}, Devroye and Sbihi~\cite{DS1990}, Biggs~\cite{B1993}, Carmona, Encinas, and Jiménez~\cite{cej2022}), as well as results for $f$-equitable graphs (such as Nishimura~\cite{N2023}). Moreover, we strengthen Rao’s theorem~\cite{R2013}, extending its applicability to additional graph families, including cone graphs and graphs arising from symmetric association schemes, and along the way we correct an error in the original statement of \cite[Theorem 2.1]{R2013}.
}

\textcolor{black}{Our methodology is based in a novel connection between maximal-entropy random walks and the spectral theory of the socalled weight-equitable partitions of graphs. The latter has proven to be a powerful theory for extending a variety of spectral results (see, e.g., Abiad~\cite{A2019}, Abiad, Hojny, and Zeijlemaker~\cite{AHZ2022}, Fiol~\cite{F1999}, Fiol and Garriga~\cite{FG1999}, Haemers~\cite{H1995}, Lee and Weng~\cite{LW2012}). This paper thus presents yet another application of the mentioned spectral machinery. We demonstrate the effectiveness of our algebraic-combinatorial framework by explicitly computing the average hitting times for several graph classes exhibiting certain regularity properties.
}

This paper is organized as follows.
In Section \ref{sec:pre}, we provide the definition of $f$-equitable graphs and discuss some of their properties related to simple random walks. In Section \ref{sec:weightequitable}, we propose \revC{}{a} general algebraic \revC{ccombinatorial}{combinatorial} framework to obtain the average hitting times for classes \revC{}{of} graphs having high regularity. As a tool\revC{}{,} we use weight-equitable partitions of graphs. 
In particular, Theorems~\ref{thm:equiW} and~\ref{thm:equi} are the main results of this paper.
In Section \ref{sec:pseudodistanceregularized}\revC{}{,} we illustrate multiple applications of the proposed method to derive the average hitting time for several graph classes such as distance-regularized graphs, $t$-distance-regular graphs\revC{}{,} and cone graphs.

Our results suggest that weight-equitable partitions could play an important role in the analysis of the maximal entropy random walk.

%%%%%%%%%%%%%%%%%%%%%%%%%%%%%%%%%%%%%%%%%%%%%%%%%%%%%%%%%%%%%%%%%%%%%%%%%%%%%%%%%%%%%%%%%%
\section{Preliminaries}\label{sec:pre}
%%%%%%%%%%%%%%%%%%%%%%%%%%%%%%%%%%%%%%%%%%%%%%%%%%%%%%%%%%%%%%%%%%%%%%%%%%%%%%%%%%%%%%%%%%

%%%%%%%%%%%%%%%%%%%%%%%%%%%%%%%%%%%%%%%%%%%%%%%%%%%%%%%%%%%%%%%%%%%%%%%%%%%%%%%%%%%%%%%%%%
\subsection{Simple random walk and maximal entropy random walk}
%%%%%%%%%%%%%%%%%%%%%%%%%%%%%%%%%%%%%%%%%%%%%%%%%%%%%%%%%%%%%%%%%%%%%%%%%%%%%%%%%%%%%%%%%%

In this paper, we consider finite and connected graphs. We denote the identity matrix as $I$, all-one vector as $\bold{\allone}$, 
\revA{$v_x$ as the $x$th entry of vector $\bold{v},$ and $A_{-x}$ as the submatrix obtained by removing the $x$th row and the $x$th column from $A$.
Let $A=(a_{ij})$ be a square matrix \revC{with}{of} size $n$.}
{$v_x$ as the $x$th entry of vector $\bold{v}.$
Let $A=(a_{ij})$ be a square matrix with size $n$.
We define $A_{-x}$ as the submatrix obtained by removing the $x$th row and the $x$th column from $A$.}

Define $T(A)$ as a matrix whose $ij$-entry is $\frac{a_{ij}}{\sum_{l=1}^{n}a_{lj}}$, and \revC{}{define $\bold{H}(A)$ as a \emph{hitting time vector of $A$} as follows:}
\[
  \bold{H}(A)=-\bold{\allone}(A-I)^{-1}.
\]

\begin{df}[Simple random walk]
  Let $G$ be a graph with adjacency matrix $A$.
  A \emph{simple random walk} on a graph is a stochastic process that moves randomly from vertex to vertex along the edges whose transition probability is $T(A)$.
\end{df}

Note that we sometimes consider the case when \revC{}{the} graph has \revC{a}{} loop edges, directed edges, and \revC{has}{} a weight \revC{in}{on} edges.
\revC{Although in}{In} these cases, \revC{}{the} graph is always connected.

Next, we define the other random walk called maximal entropy random walk.

Let $G$ be a graph with adjacency matrix $A$. 
Since~ we assume graphs to be connected (so~$A$ is irreducible), the Perron-Frobenius Theorem assures that~$\lambda_{1}$ is simple, positive, and has a positive eigenvector. 
  Moreover, it is the only eigenvalue that admits such an eigenvector. 
  We denote this eigenvector by~$\revC{\nu}{\boldsymbol{\nu}}=(\nu_{1},\ldots,\nu_{n})^{\top}$ and refer to it as the \emph{Perron eigenvector}. 
  It is assumed to be normalized such that~$\revC{\norm{\nu}}{\norm{\boldsymbol{\nu}}{}} = 1$. 
  For example, if~$G$ is regular, we have~$\revC{\nu}{\boldsymbol{\nu}}=\frac{1}{\sqrt{n}}\j$.
  Maximal entropy random walk is defined by the Perron eigenvector.

  \begin{df}[Maximal entropy random walk]
    Let $G$ be a graph with adjacency matrix $A$ and let $\revC{\nu}{\boldsymbol{\nu}}=(\nu_{1},\ldots,\nu_{n})^{\top}$ be a Perron eigenvector of $A$.
    Define $D_\nu=\diag(\nu_1,\ldots,\nu_n)$.
    Then, a \emph{maximal entropy random walk} on a graph $G$ is a stochastic process that moves randomly from vertex to vertex along the edges whose transition probability is $T(D_\nu A)$.
  \end{df}

  Maximal entropy random walks were introduced \revC{in \cite{BDLW2009}}{by Burda, Duda, Luck, and Waclaw~\cite{BDLW2009}} and have attracted a lot of attention since then, \revC{see e.g. \cite{DL2011,D2012,NLWY2018,OB2013,SGLNL2011}}
  {see e.g. Delvenne and Libert~\cite{DL2011}, Duda~\cite{D2012}, Niu, Liu, Wang and Yan~\cite{NLWY2018}, Ochab and Burda~\cite{OB2013}, and Sinatra, G\'omez-Garde\~nes, Lambiotte, Nicosia, and Latora~\cite{SGLNL2011}}.

\begin{rem}
    Note that a maximal entropy random walk in some graphs is often equivalent to a simple random walk.
    Let $G$ be a graph with adjacency matrix $A$, and let $D=\diag(\deg(v_1),\ldots,\deg(v_n))$.
    The transition probability matrix of a simple random walk on $G$ is given by $T(A) = AD^{-1}$,
while the transition probability matrix of a maximal entropy random walk on $G$ is $T(D_\nu A) = \frac{1}{\lambda_1} D_\nu A D_\nu^{-1}$.
    Therefore, when $\frac{1}{\lambda_1}D_\nu AD_\nu^{-1}=AD^{-1}$, the simple random walk and maximal entropy random walk act as the same stochastic process. 
    For example, \revC{}{a} regular graph holds the above equation.
\end{rem}

From these random walks, the average hitting times are also defined.

\begin{df}[The average hitting time]
  The \emph{hitting time} by the simple (maximal entropy) random walk from vertex $u$ to vertex $v$ is defined as the number of steps required to reach vertex $v$ for the first time in a simple (maximal entropy) random walk starting from vertex $u$.
  The expected value of the hitting time from vertex $u$ to vertex $v$ is called the \emph{average hitting time} by the simple(maximal entropy) random walk from $u$ to $v$. 
\end{df}
We denote the average hitting times from $u$ to $v$ by simple or maximal entropy random walk as $H(G;(v,u))$ and $H_m(G;(v,u))$, respectively. 
We also define $H(G;(v,v)) = H_m(G;(v,v)) = 0$.

It is known that 
\begin{equation}\label{eq:ht}
H(G;(v,u))={\bold{H}(T(A)_{-v})}_u.  
\end{equation}
However, in general, if the adjacency matrix of a graph is not known, then an explicit formula for the average hitting times is also unknown. 
Rao~\cite{R2013} showed that if a graph satisfies a certain condition related to automorphisms, then the average hitting time from a specific vertex to another can be computed without using the adjacency matrix. 
Moreover, in the case of distance-regular graphs, the average hitting times are determined by their intersection numbers (\revC{\cite{B1993,DS1990,N2023,S1984}}{Van Slijpe~\cite{S1984}, Devroye and Sbihi~\cite{DS1990}, Biggs~\cite{B1993}, and Nishimura~\cite{N2023}}).

%%%%%%%%%%%%%%%%%%%%%%%%%%%%%%%%%%%%%%%%%%%%%%%%%%%
\subsection{(Weight) equitable partitions}
%%%%%%%%%%%%%%%%%%%%%%%%%%%%%%%%%%%%%%%%%%%%%%%%%%%
Let $G$ be a graph and let $P=\{V_0,V_1,\ldots,V_{r}\}$ be a partition of the vertex set of $G$.

\begin{df}[Equitable partition]\label{def:equitable}
  The partition $P$ is called an \textit{equitable partition} if, for all integers $i$ and $j$ such that $0\leq i,j\leq r$ and for all vertices $u$ and $v$ in $V_j$, 
  $|G(u)\cap V_i|=|G(v)\cap V_i|$ holds.
\end{df}

\begin{df}[Quotient matrix, Quotient graph]
  Let $P$ be an equitable partition of $G$.
  \revA{A matrix $b_{ij}=|G(v)\cap V_i|$, where $v$ is for any vertex in $V_j$, }
  {A matrix $B=(b_{ij})$ whose $(i,j)$-entry is $|G(v)\cap V_i|$, where $v$ is an arbitrary vertex in $V_j$, }
  is called a \textit{quotient matrix} of $P$. Furthermore, the graph whose vertex set is $P$ and whose adjacency matrix is a quotient matrix is called the \textit{quotient graph} of $P$.
\end{df}
Note that the quotient graph is a directed graph whose edges may have weights and may also include loops\revA{.}
{, in which the weight of the edge $(j,i)$ is given by $b_{ij}$.}

\begin{ex}
For example, let $G$ be the graph shown on the left of Figure~\ref{fig:cube}, and suppose the vertex subsets are  
$V_0=\{v_0\}$, $V_1=\{v_1,v_2,v_4\}$, $V_2=\{v_3,v_5,v_6\}$, and $V_3=\{v_7\}$, with the partition  
$P=\{V_0,V_1,V_2,V_3\}$. Then $P$ is equitable, and the quotient graph of $P$ is the graph shown on the right of Figure~\ref{fig:cube}.

    \begin{figure}[t]
      % \centering
\begin{tikzpicture}
            % 頂点の座標
          \coordinate (A) at (1,1);
          \coordinate (B) at (2,1);
          \coordinate (C) at (2,2);
          \coordinate (D) at (1,2);
          \coordinate (E) at (0,0);
          \coordinate (F) at (3,0);
          \coordinate (G) at (3,3);
          \coordinate (H) at (0,3);
        
          % 線で頂点を結んで立方体を描画
          \draw (A) -- (B) -- (C) -- (D) -- cycle; % 底面
          \draw (E) -- (F) -- (G) -- (H) -- cycle; % 上面
          \draw (A) -- (E);
          \draw (B) -- (F);
          \draw (C) -- (G);
          \draw (D) -- (H);
        
          % 頂点のラベル
          \fill[black] (D) circle (2pt) node[left,font=\small] {$v_3$};
          \fill[black] (E) circle (2pt) node[left,font=\small] {$v_0$};
          \fill[black] (C) circle (2pt) node[right,font=\small] {$v_7$};
          \fill[black] (A) circle (2pt) node[left,font=\small] {$v_1$};
          \fill[black] (B) circle (2pt) node[right,font=\small] {$v_5$};
          \fill[black] (F) circle (2pt) node[right,font=\small] {$v_4$};
          \fill[black] (G) circle (2pt) node[right,font=\small] {$v_6$};
          \fill[black] (H) circle (2pt) node[left,font=\small] {$v_2$};

          \node at (5,1.5) {$\Longrightarrow$};

        \begin{scope}[shift={(7,1.5)},scale=2]            

  % 頂点の定義（90度回転: (x,y) → (y,-x)）
  \coordinate (E) at (0,0);
  \coordinate (A) at (1,0);
  \coordinate (D) at (2,0);
  \coordinate (C) at (3,0);

  % 有向辺と重みの追加（左右が上下に変わることに注意）
  \draw[->] (E) to [bend right,thick] node[below] {3} (0.95,0);
  \draw[->] (A) to [bend right,thick] node[below] {2} (1.95,0);
  \draw[->] (D) to [bend right,thick] node[below] {1} (2.95,0);
  \draw[->] (A) to [bend right,thick] node[above] {1} (0.05,0);
  \draw[->] (D) to [bend right,thick] node[above] {2} (1.05,0);
  \draw[->] (C) to [bend right,thick] node[above] {3} (2.05,0);

  % 頂点の描画（左→下に変化）
  \fill (D) circle (1pt) node[below,font=\small] {$V_2$};
  \fill (E) circle (1pt) node[below,font=\small] {$V_0$};
  \fill (C) circle (1pt) node[below,font=\small] {$V_3$};
  \fill (A) circle (1pt) node[below,font=\small] {$V_1$};

        \end{scope}

      \end{tikzpicture}      
      \caption{Graph $G$ and \revC{}{a} quotient graph.}
      \label{fig:cube}
    \end{figure}
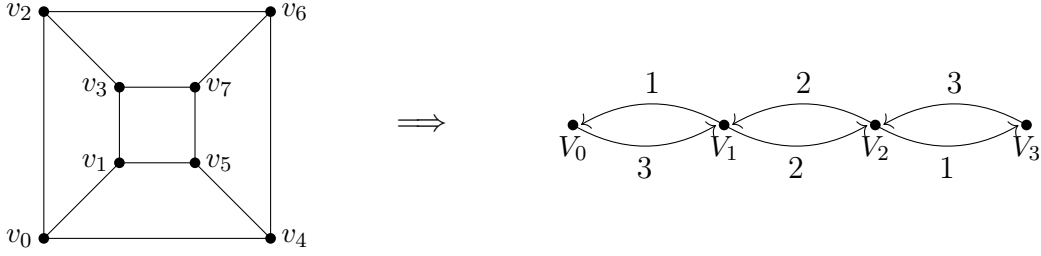
\end{ex}

Let $G$ be a graph with an equitable partition $P$ and quotient matrix $B$.  
We denote by $G_B$ the quotient graph of $P$.  
We define the simple random walk on $G_B$ as the stochastic process that moves randomly from vertex to vertex along edges, where the transition probability is given by $T(B)$.  
We also denote by \revA{$H(G_B; V_i, V_j)$}{$H(G_B; (V_i, V_j))$} the average hitting time from $V_j$ to $V_i$ in the simple random walk on $G_B$.
To facilitate the calculation of the average hitting times in certain graphs, Nishimura introduced a class of graphs called $f$-equitable graphs \cite{N2023}.

  \begin{df}[$f$-equitable graph \cite{N2023}]
    Let $G$ be a connected graph with the vertex set $V$, and let $f:V\times V\mapsto \{c,x_1,\ldots,x_r\}$ be a function.
    Define $F_{x_i}(o)=\{v\in V\mid f(o,v)=x_i\}$ as a subset of $V$, and let 
    \[
      P_o=\{F_{c}(o),F_{x_1}(o),\ldots,F_{x_r}(o)\}
    \] 
    be the partition of $V$, where $o$ is any vertex in $V$.
    $G$ is called an \textit{$f$-equitable graph} if, for all vertices $o$ in the graph $G$, $P_o$ satisfies the following three conditions:
   \begin{itemize}
    \item $F_{c}(o)=\{o\}$;
    \item $P_o$ is equitable;
    \item There exists a graph $G_B$ that is isomorphic to the quotient graph of $P_o$ for every $o$.
   \end{itemize}
  \end{df}

By definition, an $f$-equitable graph has a unique quotient graph.  
We refer to this unique graph as the quotient graph of the $f$-equitable graph. 

\begin{ex}
For example, let $G$ be the graph shown on the left of Figure~\ref{fig:cube}, with vertex set $V=\{v_0,\ldots,v_7\}$.  
Define the function $d_G: V \times V \to \{0,1,2,3\}$ by letting $d_G(u,v)$ denote the distance between vertices $u$ and $v$ in $G$. Then $G$ is a $d_G$-equitable graph.
\end{ex}

It is known that the average hitting times on an $f$-equitable graph coincide with those on its quotient graph\revC{:}{.}

  \begin{thm}[\cite{N2023}]\label{thm:equi_old}
  Let $G$ be an $f$-equitable graph with the quotient graph $G_B$.

     Then, for any vertices $u$ and $v$ in $G$,
    \[
    H(G;(v,u))=H(G_B;(c,v_u)),
    \]
    where $v_u=f(v,u)$.
  \end{thm}
From Theorem~\ref{thm:equi_old}, in some cases, the average hitting times can be computed without full information about the adjacency matrix.
For example, in a distance-regular graph, the average hitting times can be determined solely from the intersection array~\revC{\cite{B1993,DS1990,N2023,S1984}}{by Van Slijpe~\cite{S1984}, Devroye and Sbihi~\cite{DS1990}, Biggs~\cite{B1993}, and Nishimura~\cite{N2023}}.  
Other examples are discussed \revC{in~\cite{N2023}}{by Nishimura~\cite{N2023}}.
Theorem \ref{thm:equi_old} is proved using the properties of a specific equitable partition, which we detail below.

\begin{df}[Stabilized equitable partition]\label{def:stabilizedequitablepartition}
  Let $G$ be a graph and let $P=\{V_0,V_1,\ldots,V_r\}$ be its equitable partition. Then
  $P$ is called a \emph{stabilized equitable partition} centered at vertex $o$ if $V_0=\{o\}$.
\end{df}

\begin{thm}[\cite{N2023}]\label{thm:stabHt}
  Let $G$ be a graph and $P=\{V_0=\{o\},V_1,\ldots,V_r\}$ denote its stabilized equitable partition centered on vertex $o$, with quotient graph $G_B$.
  For any vertex $v$ in $V_i$,
  \[
    H(G;(o,v))  = H(G_B;(V_0,V_i)).
  \]
\end{thm}

In order to extend the above tools to a more general context\revC{}{,} we will use the \revC{socalled}{so-called} weight-equitable partitions, which are \revC{}{a} natural generalization of equitable partitions.  Such weight partitions use the entries of the Perron eigenvector of the adjacency matrix to assign weights to the vertices. Doing so, we ``regularize'' the graph, in the sense that the weight-degree of each vertex becomes a constant. The theory of weight partitions of a graph has been shown to be a very powerful method to extend several results of spectral nature to general graphs, 
\revC{see e.g.~\cite{A2019,AZ2024,F1999,FG1999,H1995,LW2012}.}
{see, e.g., Abiad~\cite{A2019} or Fiol~\cite{F1999}.}

The adjacency eigenvalues of $A$ are denoted by \revC{$\lambda_1\geq \cdots \geq \lambda_n$}{$\lambda_1,\ldots,\lambda_n$ with $\lambda_1\geq \cdots \geq \lambda_n$} and $\revC{\nu}{\boldsymbol{\nu}}=(\nu_{1},\ldots,\nu_{n})^{\top}$ is \revC{}{the} Perron eigenvector of $A$. In a~$k$-regular graph, the rows of its adjacency matrix sum up to its largest eigenvalue,~$\lambda_1=k$. If~$G$ is not regular, we can ``regularize'' it by assigning \revC{the}{a} weight to each vertex~$u\in V$. 
This ensures that the \emph{weight-degree} $\delta_u^{\ast}$ of each vertex $u\in V$ equals the largest eigenvalue, i.e.,
        \[
          \delta^{*}_{u}
          \define
          \frac{1}{\nu_{u}}\sum_{v\in G (u)}\nu_{v}
          =
          \lambda_{1}.
        \]	 
If $\revA{\mathcal{P}}{P}$ is a partition of the vertex set $V=V_{1}\cup \cdots \cup V_{m}$, the \emph{weight-intersection number}
        of $u\in V_{j}$, \revC{}{for} $j \in \revC{[m]}{\{1,\ldots,m\}}$, is
        \begin{align*}
          b^{*}_{ij}(u)
          &\define
          \frac{1}{\nu_{u}}\sum_{v \in G(u)\cap V_{i}}\nu_{v},
          && i,j \in \revC{[m]}{\{1,\ldots,m\}}.
        \end{align*}
\begin{df}[Weight-equitable]
Let $G$ be a graph and let $\revA{\mathcal{P}}{P}$ be a partition of $G$.
$\revA{\mathcal{P}}{P}$ is \emph{weight-equitable} (or \emph{weight-regular}) if~$b^*_{ij}(u) = b^*_{ij}(v)$ for all~$i,j \in \revC{[m]}{\{1,\ldots,m\}}$ and~$u,v \in V_i$.
In other words, the weight-intersection numbers are independent of the choice of~$u \in V_j$.
\end{df}

\begin{df}[Weight-regular quotient matrix, weight-quotient graph]
  Let $G$ be a graph and let $\revA{\mathcal{P}}{P}$ be a weight-equitable partition of $G$.
  We write~$b^*_{ij}$ instead of~$b^{*}_{ij}(u)$, $u \in V_j$, and we call $B^* = (b_{ij}^*)$ the \emph{weight-regular quotient matrix} (or \emph{weight-equitable quotient matrix}) with respect to $\revA{\mathcal{P}}{P}$ (note that while in \revC{\cite{A2019}}{Abiad~\cite{A2019}}, $B^*$ is called the weight-regular quotient matrix,  in \revC{\cite{F1999}}{Fiol~\cite{F1999}} the same matrix is called \emph{pseudo-quotient matrix}).
We also define the \emph{weight-quotient graph} of a weight-equitable partition as the directed graph whose adjacency matrix is $B^*$.
\end{df}

\begin{rem}
Let $G$ be a graph and let $P$ be a (weight-)equitable partition of $G$.  
In this paper, we sometimes allow $P$ to include the empty set.  
However, we define the vertex set of the (weight-)quotient graph of $P$ to exclude any empty sets.

\begin{ex}
For example, let $G$ be the graph shown on the left of Figure~\ref{fig:cube}, and define $d_G$ as the distance between two vertices in $G$, and
\[
  D_i(v_0) \coloneqq \{ u \in V \mid d_G(v_0,u) = i \}.
\]  
Then, although $D_4(v_0) = \emptyset$, we still consider 
\[
  P = \{ D_0(v_0), \ldots, D_3(v_0), D_4(v_0) \}
\] 
to be a partition of $G$.  
Even in this case, we call $P$ a (weight-)equitable partition of $G$.  
However, the quotient graph of $P$ is still the graph shown on the right of Figure~\ref{fig:cube}.
\end{ex}

From the above definition, it follows that when $G$ is connected, the (weight-)quotient graph of $P$ is always connected.
\end{rem}

The following known result, which we will use later on, establishes the relationship between the entries of the quotient matrix $Q$ and \revC{}{the} weight-regular quotient matrix $B^\ast$.

  \begin{thm}[\cite{A2019}]\label{thm:eqRw}
  A partition $P = \{V_1, V_2, ..., V_m\}$ of a graph $G$ is equitable if and only if
    it is weight-equitable and the map on $V$, denoted $\rho : u\longrightarrow\nu_u$, is constant over each \revC{$V_i (1\leq i\leq m)$}{$V_i \,(\text{for } 1\leq i \leq m)$}. 
    Then, it holds that the quotient matrix entries of the equitable partition $b_{ij}$
    and the weight-regular quotient matrix of the weight-equitable partition $b^*_{ij}$ satisfy
  \[
    b^*_{ij}=\frac{\nu_j}{\nu_i}b_{ij}.
  \]
  \end{thm}

 Note that an equitable partition is weight-equitable, but the converse is not always true. 
 For more background on weight-equitable partitions, \revC{see~\cite{A2019}.}{see Abiad~\cite{A2019}.}

%%%%%%%%%%%%%%%%%%%%%%%%%%%%%%%%%%%%%%%%%%%%%%%%%%%%%%%%%%%%%%%%%%%%%%%%%%%%%%%%%%%%%%%%%%
\section{The algebraic combinatorial framework}\label{sec:weightequitable}
%%%%%%%%%%%%%%%%%%%%%%%%%%%%%%%%%%%%%%%%%%%%%%%%%%%%%%%%%%%%%%%%%%%%%%%%%%%%%
In this section, we propose a general algebraic framework that we will later use to obtain the average hitting times of several graph classes holding several regularity properties.

Instances of such highly regular graph classes are pseudo-distance-regular graphs, which were introduced by 
Fiol \cite{F2001,F2012} to define \revC{distance regularity}{distance-regularity} even for non-regular graphs. Since pseudo-distance-regular graphs are not $f$-equitable graphs, the average hitting time for these graph classes does not follow from the existing frameworks to derive average hitting times for distance-regular graphs \revC{\cite{B1993,N2023}}{such as Biggs~\cite{B1993} and Nishimura~\cite{N2023}}. 
To overcome this difficulty, in this section we focus on deriving a new theoretical framework for the \revC{socalled}{so-called} weakly weight-$f$-equitable graphs, and we show that the average hitting time of \revC{these}{this} graph class can be determined by their weight-intersection numbers (Theorem~\ref{thm:equiW}). 
Note that an $f$-equitable graph is always a weakly weight-$f$-equitable graph, but the converse does not hold.  
Furthermore, we define two specific types of weakly weight-$f$-equitable graphs as \emph{weakly $f$-equitable} and \emph{weight-$f$-equitable}.  
We can calculate the average hitting times in the former graph using two types of random walks: the simple random walk and the maximal entropy random walk (Theorem~\ref{thm:equi}). 
This allows us to clearly observe the difference between these two walks.  
We also show that a graph $G$ is weight-$f$-equitable if and only if $G$ is $f$-equitable, which implies that the latter generalization is not meaningful.
To show this result\revC{}{,} we use \revC{of}{} the theory of weight-equitable partitions of graphs, for more details\revC{}{,} we refer the reader to \revC{\cite{F1999, A2019}}{Fiol~\cite{F1999} and Abiad~\cite{A2019}}. 
We should note that, unlike previous work on the average hitting time of specific graphs which \revC{use}{used} the properties of graph automorphisms (\revC{see e.g. \cite{R2013, T2024}}{Rao~\cite{R2013} and Tanaka~\cite{T2024}}), our approach only requires \revC{to use}{using} the weight-intersection numbers.  

This section is structured as follows. In Section \ref{subsec:wpartitions}\revC{}{,} we show that if a graph $G$ has a stabilized weight-equitable partition $P$ centered at vertex $o$, then the average hitting times of the maximal entropy random walk from vertex $v$ to $o$ can be calculated using the weight-regular quotient matrix of $P$ (Theorem \ref{thm:stabHtW}). 
In Section \ref{subsec:weaklyfequitable}\revC{}{,} we define weakly weight-$f$-equitable graphs, to which we can apply Theorem \ref{thm:stabHtW} for any vertex. Additionally, we consider certain restrictions of weakly weight-$f$-equitable graphs, which we call \emph{weakly $f$-equitable graphs} and \emph{weight-$f$-equitable graphs}. 
In the former case, the average hitting times of two random walks on a weakly $f$-equitable graph can be computed via its quotient graph (Theorem~\ref{thm:equi}). 

In the latter case, we show that every weight-$f$-equitable graph is, in fact, an $f$-equitable graph (Theorem~\ref{thm:equiWisR}). 
Hence, this case has already been considered in Nishimura~\cite{N2023}.

%%%%%%%%%%%%%%%%%%%%%%%%%%%%%%%%%%%%%%%%%%%%%%%%%%%%%%%%%%%%%%%%%%%%%%%%%%%%%%%%%%%%%%%%%%
\subsection{Weight-equitable partitions and maximal entropy random walk}\label{subsec:wpartitions}
%%%%%%%%%%%%%%%%%%%%%%%%%%%%%%%%%%%%%%%%%%%%%%%%%%%%%%%%%%%%%%%%%%%%%%%%%%%%%%%%%%%%%%%%%%

Now we are ready to extend Theorem \ref{thm:stabHt} to maximal entropy random walks, see Theorem \ref{thm:stabHtW}. For this, we need first the following definition.

\begin{df}[Stabilized weight-equitable partition]\label{df:stabW}
  Let $G$ be a graph and let $P=\{V_0,V_1,\ldots,V_r\}$ be its weight-equitable partition.
  Then $P$ is called a \emph{stabilized weight-equitable partition} centered at vertex $o$ if $V_0=\{o\}$.
\end{df}

\begin{thm}\label{thm:stabHtW}
  Let $G$ be a graph, $P=\{V_0=\{o\},V_1,\ldots,V_r\}$ denote its stabilized weight-equitable partition centered on vertex $o$ with weight-quotient graph $G_{B^*}$.
  For any vertex $v$ in $V_i$,
  \[
    H_m(G;(o,v))  = H(G_{B^*};(V_0,V_i)).
  \]
\end{thm}

\begin{proof}
  Let $X_0 = x$ be the starting vertex of a maximal entropy random walk on $G$, and let $X_t$ be a random variable representing the vertex reached after $t$ steps in the random walk.
  Let us denote $Pr(X_t \in V_i | X_{t-1} \in V_j)$ as the probability of $X_t$ being in $V_i$ given that $X_{t-1}$ is in $V_j$.
  Then,
  \begin{align*}
    Pr(X_t \in V_i | X_{t-1} \in V_j)&=\frac{\sum_{v\in G(X_{t-1})\cap V_i}\nu_v}{\sum_{v\in G(X_{t-1})}\nu_v}\\
    &=\frac{\frac{1}{\nu_{X_{t-1}}}\sum_{v\in G(X_{t-1})\cap V_i}\nu_v}{\frac{1}{\nu_{X_{t-1}}}\sum_{v\in G(X_{t-1})}\nu_v}\\
    &=\frac{b_{ij}^*}{\lambda_1}.
  \end{align*}
  This implies that the transition probability matrix of the stochastic process, where $X_t$ is in the block of $P$, is $\frac{1}{\lambda_1}(B^*)^\top$.
  Hence, similar to the \revC{disucussion}{discussion} in the proof of Theorem \ref{thm:stabHt}, the expected value of the number of steps required for $X_t$ to be in $V_0$ for the first time when $X_0\in V_i$ is  $H(G_{B^*};(V_0,V_i))$. 
  Here, since $V_0=\{o\}$, $X_t$ is in $V_0$ if and only if $X_t=o$, which means
  \[
    H_m(G;(o,v))  = H(G_{B^*};(V_0,V_i)).\qedhere
  \]
\end{proof}

%%%%%%%%%%%%%%%%%%%%%%%%%%%%%%%%%%%%%%%%%%%%%%%%%%%%%%%%%%%%%%%%%%%%%%%%%%%%%%%%%%%%%%%%%%
\subsection{Weakly weight-$f$-equitable graphs}\label{subsec:weaklyfequitable}
%%%%%%%%%%%%%%%%%%%%%%%%%%%%%%%%%%%%%%%%%%%%%%%%%%%%%%%%%%%%%%%%%%%%%%%%%%%%%%%%%%%%%%%%%%
 
Here, we introduce weakly weight-$f$-equitable graphs, to which Theorem \ref{thm:stabHtW} can be applied \revC{for}{to} all vertices.

\begin{df}[Weakly weight-$f$-equitable graph]\label{df:wwfe}
  Let $G$ be a graph with vertex set $V$, and let $f:V\times V\rightarrow \{c_1,c_2,\ldots,c_s,x_1,\ldots,x_r\}$ be a function such that
  \[
    f(v,u)\in \begin{cases}
      \{c_1,\ldots,c_s\}&u=v\\
      \{x_1,\ldots,x_r\}&u\neq v
    \end{cases}.
  \]
  Define $F_{x_i}(o)=\{u\in V\mid f(o,u)=x_i\}$ as a subset of $V$, and let 
  \[
    P_o=\{F_{f(o,o)}(o),F_{x_1}(o),\ldots,F_{x_r}(o)\}
  \] 
  be the partition of $V$, where $o$ is any vertex in $V$.
 A graph $G$ is called a \textit{weakly weight-$f$-equitable graph} if, for all vertices $o$ in  $G$, $P_o$ satisfies the following three conditions:
 \begin{itemize}
  \item $F_{f(o,o)}(o)=\{o\}$;
  \item $P_o$ is weight-equitable;
  \item  The weight-quotient graph of $P_o$ only \revC{depend}{depends} on the value of $f(o,o)$.
 \end{itemize}
\end{df}

Let $G$ be a weakly weight-$f$-equitable graph, where 
$f: V \times V \rightarrow \{c_1, c_2, \ldots, c_s, x_1, \ldots, x_r\}$ is a function such that
\[
  f(u,v) \in 
  \begin{cases}
    \{c_1, \ldots, c_s\} & \text{if } u = v, \\
    \{x_1, \ldots, x_r\} & \text{if } u \neq v.
  \end{cases}
\]  
We call the set $\{c_1, \ldots, c_s\}$ the \emph{color set} of $f$, and the set $\{x_1, \ldots, x_r\}$ the \emph{relation set} of $f$.  
An element of the color set is referred to as a \emph{color}.  
Then, $G$ has $s$ types of weight-regular quotient matrices and weight-quotient graphs corresponding to the elements of the color set.  
We denote the weight-regular quotient matrix and the weight-quotient graph corresponding to a color $c$ by $B^*(c)$ and $G^*_c$, respectively.

In the next result we show that when $G$ is weakly weight-$f$-equitable, the average hitting times of the maximal entropy random walk on $G$ are equal to those of the simple random walk on the corresponding weight-quotient graphs.

\begin{thm}\label{thm:equiW}
 Let $G$ be a weakly weight-$f$-equitable graph with the weight-quotient graphs $\{G^*_c\}_{c\in C}$, where $C$ is a color set of $f$.
  Then, for any vertices $u$ and $v$ in $G$,
  \[
    H_m(G;(v,u))=H(G^*_{c_v}; (c_v,v_u)),
  \]
  where $c_v=f(v,v)$ and $v_u=f(v,u)$.
\end{thm}
\begin{proof}
Let $R$ be the relation set of $f$.
  For any vertex $v$ in $G$, there exists $\{x_{v_1},\ldots,x_{v_r}\}\subset R$ such that 
  \[
    P_v=\{F_{c_v}(v),F_{x_{v_1}}(v),\ldots,F_{x_{v_r}}(v)\}
  \]
  be a partition where all of the blocks are non-empty sets.
  From the definition of weakly weight-$f$-equitable, $P_v$ is a stabilized weight-equitable partition centered on $v$, and its weight-quotient graph is $G_{c_v}^*$. 
  Therefore, from Theorem \ref{thm:stabHtW}, for any vertex $u\in F_{x_{v_i}}(v)$,
    \begin{align*}
      H_m(G,(v,u))=H(G^*_{c_v}; (c_v,v_u)).      \,\,\,
    \end{align*}\qedhere 
\end{proof}
  From Theorem \ref{thm:equiW}, we obtain an alternative proof of Theorem \ref{thm:equi_old}.
  \begin{proof}[Alternative proof of Theorem \ref{thm:equi_old}]
Let $G$ be any $f$-equitable graph.
We can also consider $G$ as a weakly weight-$f$-equitable graph, where the color set of $f$ is a singleton.  
Therefore, from Theorem~\ref{thm:equiW}, we obtain
\[
  H_m(G, (v,u)) = H(G_B^*, (c, v_u)),
\]
where $G_B^*$ is the weight-quotient graph of $f$.  
Since all partitions defined by $f$ are also equitable, we can also define the quotient graph of $f$, denoted by $G_B$.
Additionally, since $G$ is $f$-equitable, $G$ must be regular, which implies $G_B^*$ is isomorphic to $G_B$ and 
    \[
      H_m(G,(v,u))=H(G,(v,u)). 
    \]
    This implies
    \[
      H(G,(v,u))=H(G_{B},(c,v_u)). \qedhere
    \]
  \end{proof}

Recall that an $f$-equitable graph is a weakly weight-$f$-equitable graph that satisfies two conditions: one, the color set must be a singleton; and two, all partitions must be equitable.  
Next, we relax this definition and consider specific weakly weight-$f$-equitable graphs in which only one of these conditions holds.  
First, we consider the case of a weakly weight-$f$-equitable graph in which all partitions generated by $f$ are equitable.

  \begin{df}[Weakly $f$-equitable]
    Let $G$ be a weakly weight-$f$-equitable graph.
    Then $G$ is \emph{weakly $f$-equitable} if all partitions generated by $f$ are an equitable partition.
  \end{df}

  \begin{rem}\label{rem:1}
Weakly $f$-equitable graphs are not always $f$-equitable. An example of such graphs is given by distance-regularized graphs, which are discussed in Section~\ref{sec:distanceregularized}.  
Additionally, there also exist graphs that are weakly weight-$f$-equitable but not weakly $f$-equitable for a fixed $f$.
    Indeed, let $G$ be a graph with vertex set $V=\{v_1,\dots,v_n\}$ and a weight-equitable partition $P=\{V_0=\{v_1\},V_1,\dots,V_m\}$ that is not equitable; see an example in \revC{\cite{F1999}}{Fiol~\cite{F1999}} of such a graph $G$ and a weight-equitable partition $P$. Then, suppose $f:V\times V\rightarrow \{c_1,c_2,\ldots,c_n,x_1,\ldots,x_n\}$ is defined by  
\[
f(v_i,v_j)=\begin{cases}
  c_i & i=j,\\
  x_k & i=1, v_j\in V_k,\\
  x_j & i\neq j, i\neq 1.
\end{cases}
\]
  Then, $G$ is weakly weight-$f$-equitable but not weakly $f$-equitable.
\end{rem}

 We denote the quotient matrix and quotient graph corresponding to a color $c$ of $f$ by $B(c)$ and $G_c$, respectively.  
Since a weakly $f$-equitable graph is also weakly weight-$f$-equitable, there exists a corresponding weight-regular quotient matrix $B^*(c)$.  
From Theorem~\ref{thm:eqRw}, we have
\[
  B^*(c) = D_\nu B(c) D_\nu^{-1}.
\]
  Using the definitions and results above, we can prove the following result.

\begin{thm}\label{thm:equi}
  Let $G$ be a weakly $f$-equitable graph with quotient matrix $B(c)$ and quotient graph $G_c$, where $c$ is a color of $f$.

Define $B^*(c)=D_\nu B(c)D_\nu^{-1}$ and let $G_c^*$ be the graph whose adjacency matrix is $B^*(c)$. 

  Then, for any vertices $u$ and $v$ in $G$,
 \begin{align*}
    H(G;(v,u))&=H(G_{c_v};(c_v,v_u))\\
    H_m(G;(v,u))&=H(G_{c_v}^*;(c_v,v_u)),
  \end{align*} 
  where $f(v,v)=c_v$ and $f(v,u)=v_u$.
\end{thm}

\begin{proof}
Let $R$ be the relation set of $f$. By the same argument as in the proof of Theorem~\ref{thm:equiW}, there exists a subset $\{x_{v_1}, \ldots, x_{v_r}\} \subset R$ such that
\[
  P_v = \{ F_{c}(v), F_{x_{v_1}}(v), \ldots, F_{x_{v_r}}(v) \}
\]
is a stabilized equitable partition centered at $v$, and its quotient graph is $G_{c_v}$. 
  Therefore, from Theorem \ref{thm:stabHt}, for any vertex $u\in F_{x_{v_i}}(v)$,
  \begin{align*}
    H(G,(v,u))=H(G_{c_v};(c_v,v_u)).
  \end{align*}
  From the definition of weakly-$f$-equitable, $G$ is also weakly weight-$f$-equitable and the weight-regular quotient matrix is $B^*(c)$.
  Therefore, we obtain the second equation by Theorem \ref{thm:equiW}.\qedhere
\end{proof}
\begin{rem}
  Note that any graph is weakly $f$-equitable. Let $G$ be a graph with vertex set $V=\{v_1,\ldots,v_n\}$, and define $f(v_i,v_j)=v_j$ if $i\neq j$ and $f(v_i,v_i)=i$.
  Then, the partition constructed by $f$ is the finest subdivision where each block contains only one vertex.
  Since such a partition is always equitable, $G$ and $f$ satisfy the condition of being weakly $f$-equitable.
  Hence, when computing the average hitting times of $G$, it is not particularly important to verify that $G$ is weakly $f$-equitable; rather, the crucial aspect is determining whether a ``good'' $f$ exists, meaning that $G$ is weakly $f$-equitable and both the number of quotient matrices and their size are small.
\end{rem}

    Next, we investigate the case when the weakly weight-$f$-equitable graph has a singleton color set.

\begin{df}[Weight-$f$-equitable graph]
  Let $G$ be a weakly weight-$f$-equitable graph.
  We call $G$ as weight-$f$-equitable if the color set of $f$ is a singleton.
\end{df}

From the definition, \revC{}{an} $f$-equitable graph is always \revC{}{a} weight-$f$-equitable graph.
Actually, the converse \revC{is}{} also holds.

\begin{thm}\label{thm:equiWisR}
  If $G$ is weight-$f$-equitable, then $G$ is $f$-equitable.
\end{thm}

\begin{proof}
  Let $G$ be a weight-$f$-equitable graph. 
  If for all $v\in V(G), \nu_v$ are constants, then \revC{it's}{it is} clear that any weight-equitable partition $P_o$ is also an equitable partition.
  Hence, we only need to prove it.
  We assume there \revC{exists}{exist} two vertices $v_0$ and $v_1$ such that $\nu_{v_0}\neq \nu_{v_1}$, and\revC{}{,} without loss of generality, we assume $\nu_{v_0}<\nu_{v_1}$.
  Next, we show $v_0$ and $v_1$ are not adjacent.
    
  Since $G$ is weight-$f$-equitable, there exists a weight-equitable partition $P_{v_0}=\{V_{x_0}(v_0),\ldots,V_{x_r}(v_0)\}$, 
  \revA{and there exist an integer $l$ such that $v_1\in V_{x_l}(v_0)$.}
  {where $\{x_0\}$ is a color set of $f$.
  Then, there exists an integer $l\neq0$ such that $v_1\in V_{x_l}(v_0)$.}
  Therefore, if $v_0$ is adjacent to $v_1$, then we have $b^*_{x_l,x_0} = \frac{\nu_{v_0}}{\nu_{v_1}} < 1$.

  Next, we consider a weight-equitable partition $P_{v_1}=\{V_{x_0}(v_1),\ldots,V_{x_r}(v_1)\}$ 
  and choose a vertex $v_2$ from $V_{x_l}(v_1)$.
  Using the same approach as before, we have:
  \[
    b^*_{x_l,x_0}=\frac{\nu_{v_1}}{\nu_{v_2}}=\frac{\nu_{v_0}}{\nu_{v_1}}.
  \] 

  Then, we can repeat this process by considering $P_{v_2}$, and there exists a vertex $v_3$ in $V_{x_l}(v_2)$.
  By continuing this process, we obtain a sequence of vertices $\{v_0,v_1,\ldots,v_k\}$,
  where all \revC{$i\in [k]$,}{$i\in \{1,\ldots,k\}$,}
  \[
    \frac{\nu_{v_{i-1}}}{\nu_{v_{i}}}=\frac{\nu_{v_0}}{\nu_{v_1}}.
  \] 
  This implies
  \[
    \frac{\nu_{v_{0}}}{\nu_{v_k}}= (\frac{\nu_{v_0}}{\nu_{v_1}})^k.
  \]
  Since this process can be repeated infinitely, 
  there exists an infinitely small $\frac{\nu_{v_{0}}}{\nu_{v_k}}$, 
  which contradicts the fact that the graph has a finite number of vertices.
  Therefore, $v_0$ is not adjacent to $v_1$.
  From the above, we obtain that if $\nu_{v_0}\neq \nu_{v_1}$, then $v_0$ is not adjacent to $v_1$.
  From the contrapositive, for any pair of vertices $v_0$ and $v_1$ \revC{which}{that} are adjacent to each other, $\nu_{v_0}=\nu_{v_1}$.
  Since $G$ is a connected graph, $\nu_v$ must be constant for all $v\in V(G)$. 
\end{proof}

%%%%%%%%%%%%%%%%%%%%%%%%%%%%%%%%%%%%%%%%%%%%%%%%%%%%%%%%%%%%%%%%%%%%%%%%%%%%%%%%%%%%%%%%%%
\section{Applications}\label{sec:pseudodistanceregularized}
%%%%%%%%%%%%%%%%%%%%%%%%%%%%%%%%%%%%%%%%%%%%%%%%%%%%%%%%%%%%%%%%%%%%%%%%%%%%%%%%%%%%%%%%%%

Here\revC{}{,} we will use the framework shown in Section \ref{sec:weightequitable} to obtain the average hitting time of several graph classes with a high degree of regularity. 
In particular, we will apply it to pseudo-distance-regular graphs, distance-regularized graphs (which include the well-known distance-regular and distance biregular graphs), $t$-distance-regular graphs and cone graphs. 
In particular, since distance-regularized graphs and cone graphs are weakly $f$-equitable, where $f$ has ``good'' properties, we can apply Theorem~\ref{thm:equiW} to compute the average hitting times in these graph classes.

%%%%%%%%%%%%%%%%%%%%%%%%%%%%%%%%%%%%%%%%%%%%%%%%%%%%%%%%%%%%%%%%%%%%%%%%%%%%%%%%%%%%%%%%%%
\subsection{The average hitting time of pseudo-distance-regular
graphs}\label{sec:pseudodrg}
%%%%%%%%%%%%%%%%%%%%%%%%%%%%%%%%%%%%%%%%%%%%%%%%%%%%%%%%%%%%%%%%%%%%%%%%%%%%%%%%%%%%%%%%%%

In this section, we will use Theorem \ref{thm:stabHtW} to obtain the average hitting time of pseudo-distance-regular graphs. This graph class was introduced by Fiol \cite{F2001} as a natural generalization of distance-regular graphs\revC{}{,} which is intended for not necessarily regular graphs. Since their introduction, this graph class has received quite some attention in the literature, \revC{see e.g.~\cite{jstb08,jsj2010}}{see, e.g., Jafarizadeh, Sufiani, Taghavi, and Barati~\cite{jstb08} and Jafarizadeh, Sufiani, and Jafarizadeh~\cite{jsj2010}}. 

\begin{df}\label{df:pdrg}
A graph $G$ is called \emph{pseudo-distance-regular} around vertex $v$ if the distance partition 
  \[
    P=\{\G_0(v)=\{v\},\G_1(v),\ldots,\G_{d}(v)\}
  \] 
  is weight-equitable.
We refer to the weight-regular quotient matrix of a pseudo-distance-regular graph around $v$ as the weight-regular quotient matrix of $P$.
\end{df}

    Let $G$ be a pseudo-distance-regular graph around $v$.
    Then, the weight-regular quotient matrix of a distance partition from $v$, denoted as $B^*$, is always a tridiagonal matrix.
    When
    \begin{align*}
    B^*&=\begin{pmatrix}
      0 & c_1^* &  & &&\\
      b_0^* & a_1^* &&&&\\
       & b_1^* &&&&\\
      &     &&\ddots &c_{d-1}^*&\\
    &&&&a_{d-1}^* &c_{d}^*\\
      &&&&b_{d-1}^*&a_{d}^*
    \end{pmatrix},
  \end{align*}
we call the array
\[
\left[\begin{array}{ccccc}
\ast		&	c_1^*	&	\cdots	&	c_{d-1}^*	&	c_{d}^*\\
0		&	a_1^*	&	\cdots	&	a_{d-1}^*	&	a_{d}^*\\
b_0^*	&	b_1^*	&	\cdots	&	b_{d-1}^*	&	\ast
\end{array}\right]
\]
the \emph{pseudo-intersection numbers}.

Since pseudo-distance-regular graphs are defined by a specific vertex property, we can apply Theorem \ref{thm:stabHtW} and obtain the following result.

\begin{thm}\label{thm:pseudo}
Let $G$ be pseudo-distance-regular around vertex $v$, and let the array
  \[
  \left[\begin{array}{ccccc}
  \ast & c_1 & \cdots & c_{d-1} & c_d \\
  0 & a_1 & \cdots & a_{d-1} & a_d \\
  b_0 & b_1 & \cdots & b_{d-1} & \ast
  \end{array}\right]
  \]
  be the pseudo-intersection numbers of $G$. Define
    \begin{align*}
    B^*&=\begin{pmatrix}
      0 & c_1^* &  & &&\\
      b_0^* & a_1^* &&&&\\
       & b_1^* &&&&\\
      &     &&\ddots &c_{d-1}^*&\\
    &&&&a_{d-1}^* &c_{d}^*\\
      &&&&b_{d-1}^*&a_{d}^*
    \end{pmatrix}.
  \end{align*}
  Then,
  \[
    H_m(G;(v,u))  = \bold{H}(\frac{1}{\lambda_1}(B^*)^\top_{-\G_0})_{\G_i},
  \]
  \revA{}{where $\G_i$ denotes the $i$-th row or column of the matrix.}
  Moreover, if $P=\{G_0(v),G_1(v),\ldots,G_{d}(v)\}$ is equitable, then
  \[
    H(G;(v,u))  = \bold{H}(T(B)_{-\G_0})_{\G_i},
  \]
  where $B=D_\nu^{-1}(B^*)^\top D_\nu$.
\end{thm}

\begin{proof}
  From \revC{the}{} Definition \ref{df:pdrg} of pseudo-distance-regular around $v$,
  it follows that $P$ is a stabilized weight-equitable partition.
  Then, using Theorem \ref{thm:stabHtW}, we obtain 
  \[
    H_m(G;(v,u))  = H(G_{B^*};(\G_0,\G_i)),
  \]
  where $G_{B^*}$ is \revC{}{the} graph whose adjacency matrix is $B^*$.
  From Equation \ref{eq:ht},  
\[
    H(G_{B^*};(\G_0,\G_i))  = \bold{H}(\frac{1}{\lambda_1}(B^*)^\top_{-\G_0})_{\G_i}.
\]
  If $P$ is an equitable partition, then the quotient matrix of $P$ is $B$, and applying Theorem \ref{thm:stabHt}, we obtain the second part of the statement. 
\end{proof}

%%%%%%%%%%%%%%%%%%%%%%%%%%%%%%%%%%%%%%%%%%%%%%%%%%%%%%%%%%%%%%%%%%%%%%%%%%%%%%%%%%%%%%%%%%
\subsection{The average hitting times of distance-regularized
graphs}\label{sec:distanceregularized}
%%%%%%%%%%%%%%%%%%%%%%%%%%%%%%%%%%%%%%%%%%%%%%%%%%%%%%%%%%%%%%%%%%%%%%%%%%%%%%%%%%%%%%%%%%

In this section\revC{}{,} we show the average hitting times of distance-regularized graphs. Godsil and Shawe-Taylor \cite{gs1987} introduced the class of distance-regularized graphs as a common generalization for distance-regular graphs and generalized polygons, and they showed that a distance-regularized graph must be either distance-regular or distance-biregular. Our results extend and unify results from \revC{\cite{B1993,cej2022,N2023}}{Biggs~\cite{B1993}, Carmona, Encinas, Jiménez~\cite{cej2022}, and Nishimura~\cite{N2023}}. 

Let $G$ be a connected graph with diameter $d$. For each $i=0,1,\dots,d$, we write $\G_i(u)$ for the set of vertices in $G$ at distance $i$  from vertex $u$. 
For any two vertices $u$ and $v$ with $v \in \G_i(u)$, we define 
\[\begin{aligned}
c(u,v)&=|\G_{i-1}(u)\cap \G_1(v)|,\\
a(u,v)&=|\G_{i}(u)\cap \G_1(v)|,\\
b(u,v)&=|\G_{i+1}(u)\cap \G_1(v)|,
\end{aligned}\]
\revC{with}{where} $c(u,v)$ \revC{}{is} undefined when $i=0$ and $b(u,v)$ \revC{}{is} undefined when $i=d$. 
We say that  vertex $u$ is \emph{distance-regularized} if the numbers $c(u,v), a(u,v)$ and $b(u,v)$ are independent of the choice of $v$ in $\G_i(u)$, and in this case\revC{}{,} we denote them by $c_{i}(u), a_{i}(u)$ and $b_{i}(u)$, respectively.
 A \emph{distance-regularized} graph is a connected graph in which every vertex is distance-regularized. 
If $G$ is distance-regularized, the array
\[
\iota(u)
=
\left[\begin{array}{ccccc}
\ast		&	c_1(u)	&	\cdots	&	c_{d-1}(u)	&	c_d(u)\\
0		&	a_1(u)	&	\cdots	&	a_{d-1}(u)	&	a_d(u)\\
b_0(u)	&	b_1(u)	&	\cdots	&	b_{d-1}(u)	&	\ast
\end{array}\right]
\]
is called the \emph{intersection array} of vertex $u$. 

Next\revC{}{,} we can use \revC{the previous}{} Theorem \ref{thm:pseudo} to compute the average hitting times in distance-regularized graphs by just using their intersection arrays.

\begin{thm}\label{thm:dbrgHT}
  Let $G$ be a distance-regularized 
   graph with \revC{a}{an} intersection array
  \[
\iota(v)
=
\left[\begin{array}{ccccc}
\ast		&	c_1(v)	&	\cdots	&	c_{d-1}(v)	&	c_d(v)\\
0		&	a_1(v)	&	\cdots	&	a_{d-1}(v)	&	a_d(v)\\
b_0(v)	&	b_1(v)	&	\cdots	&	b_{d-1}(v)	&	\ast
\end{array}\right],
\]
  and define 
  \begin{align*}
    B(v)&=\begin{pmatrix}
      0 & c_1(v) &  & &&\\
      b_0(v) & a_1(v) &&&&\\
       & b_1(v) &&&&\\
      &     &&\ddots &c_{d-1}(v)&\\
    &&&&a_{d-1}(v) &c_{d}(v)\\
      &&&&b_{d-1}(v)&a_{d}(v)
    \end{pmatrix}.
  \end{align*}
  Then,
  \begin{align*}
    H(G;(v,u))=H_m(G;(v,u))=\bold{H}(T(B(v))_{-\G_0})_{\G_i},
  \end{align*}
  where $\G_i$ denotes the $i$-th row or column of the matrix.
\end{thm}

In order to prove Theorem~\ref{thm:dbrgHT}\revC{}{,} we will use the next preliminary result, which \revC{proivdes}{provides} the maximal entropy random walk in regular and biregular graphs. Recall that a graph $G=(V,E)$ is \emph{biregular} if $G$ is bipartite with $V=V_0 \cup V_1$, where the degree of each vertex in $V_0$ and the degree of each vertex in $V_1$ are (possibly different) constants.     

\begin{lem}\label{lem:ht_eq_rht}
  Let $G$ be any regular or biregular graph with adjacency matrix $A$.
  Define $\lambda_1$ as the maximum eigenvalue of $A, \nu$ as the eigenvector corresponding to $\lambda_1, D_\nu=\diag(\nu_1,\ldots,\nu_n)$ and $D=\diag(\deg(v_1),\ldots,\deg(v_n))$.
  Then, 
  \[
    \frac{1}{\lambda_1}D_\nu AD_\nu^{-1}=AD^{-1}.
  \]
In particular, $H_m(G;v,u)=H(G;v,u)$.
\end{lem}

\begin{proof}
  If $G$ is a $k$-regular graph, then $D=kI, \lambda_1=k,$ and $\nu=\frac{1}{\sqrt{n}}\allone$, which implies $\frac{1}{\lambda_1}D_\nu AD_\nu^{-1}=AD^{-1}$.

  When $G$ is biregular, \revC{}{the} vertex set of $G$ can be denoted as $V_0 \cup V_1$, where the degree of each vertex in $V_0$ and the degree of each vertex in $V_1$ are constants.
  We denote these two constants as $k_0$ and $k_1$, such that each vertex in $V_0$ has $k_0$ neighbors and each vertex in $V_1$ has $k_1$ neighbors. 
  Then, we can write $A$ and $AD^{-1}$ as follows:

    \[
        A = \begin{pmatrix}
            0 & B \\
            B^{\top} & 0
        \end{pmatrix}, \quad AD^{-1} = \begin{pmatrix}
            0 & \frac{1}{k_1}B \\
            \frac{1}{k_0}B^{\top} & 0
        \end{pmatrix}.
    \]
  
  Since $G$ is biregular, we know (\revC{see e.g. \cite{ADF2013}}{see, e.g., Abiad, Dalf\'o, and Fiol~\cite{ADF2013}}) that the largest eigenvalue of $A$ is $\sqrt{k_0k_1}$\revC{}{,} and
  the vector $\revC{\nu}{\boldsymbol{\nu}}=(\sqrt{k_0},\ldots,\sqrt{k_0},\sqrt{k_1},\ldots,\sqrt{k_1})^{\top}$ is an eigenvector of $A$ with eigenvalue $\sqrt{k_0k_1}$. 
  This implies \[D_\nu=\frac{1}{\sqrt{|V_0|k_0+|V_1|k_1}}\begin{pmatrix}
    \sqrt{k_0}I & 0\\
    0 & \sqrt{k_1}I
  \end{pmatrix}\]
  and
  \begin{align*}
    \frac{1}{\lambda_1}D_\nu AD_\nu^{-1}&=\frac{1}{\sqrt{k_0k_1}}\begin{pmatrix}
      \sqrt{k_0}I & 0\\
      0 & \sqrt{k_1}I
    \end{pmatrix} \begin{pmatrix}
      0&B\\
      B^{\top}&0
    \end{pmatrix}\begin{pmatrix}
      \frac{1}{\sqrt{k_0}}I & 0\\
      0 & \frac{1}{\sqrt{k_1}}I
    \end{pmatrix}\\
    &=\begin{pmatrix}
      0&\frac{1}{k_1}B\\
      \frac{1}{k_0}B^{\top}&0
    \end{pmatrix}\\
    &=AD^{-1}.
  \end{align*}

 Since $\frac{1}{\lambda_1} D_\nu A D_\nu^{-1} = A D^{-1}$ holds, the transition probabilities of the maximal entropy random walk on $G$ and that of the simple random walk on $G$ are the same. This implies that the average hitting times of the maximal entropy random walk and the simple random walk are identical. 
\end{proof}

\begin{proof}[Proof of Theorem~\ref{thm:dbrgHT}]
Since it is known that a distance-regularized graph is either distance-regular or distance-biregular \revC{\cite{gs1987}}{(Godsil and Shawe-Taylor~\cite{gs1987})}, then $G$ is regular or biregular. Therefore, from Lemma \ref{lem:ht_eq_rht}, we obtain $H(G;(v,u))=H_m(G;(v,u))$.

Since $G$ is distance-regularized, $G$ is also pseudo-distance-regular around vertex $v$\revC{}{,} and $B(v)$ is a quotient matrix of $G$ around vertex $v$.
  Therefore, from Theorem \ref{thm:pseudo} and \revA{Equation \ref{eq:ht}}{Equation (\ref{eq:ht}) in Section~\ref{sec:pre}}, it follows that
  \begin{align*}
   H(G;(v,u))  = \bold{H}(T(B(v))_{-\G_0})_{\G_i}.
& \qedhere \end{align*}
\end{proof}

\begin{rem}
  We can also obtain Theorem \ref{thm:dbrgHT} from the fact that distance-regularized graphs are weakly $f$-equitable.
  We define $d_G(v,u)$ as the distance between vertex $u$ and vertex $v$ in $G$, and let $D_v=\max\{d_G(v,u) \mid u \in V\}$. 
  We also define $d_G^+(u,v): V \times V \rightarrow \{1, \ldots, D_u\} \cup V$ as a function such that if $u = v$, then $d_G^+(v,u) = v$; otherwise, $d_G^+(v,u) = d_G(v,u)$.  
  From this definition, $G$ is distance-regularized if and only if $G$ is weakly $d_G^+$-equitable.
  Therefore, using Theorem \ref{thm:equiW}, we can obtain Theorem \ref{thm:dbrgHT}.
\end{rem}

\begin{rem}
  Because a distance-regularized graph is either distance-regular or distance-biregular, there are at most two types of $B(v)$.
  If $B(v)$ does not depend on the choice of $v$, then $G$ is a distance-regular graph.
  Otherwise, $G$ is a distance-biregular, and there exist matrices $B_1$ and $B_2$ such that $B(v)=B_1$ or $B(v)=B_2$.
\end{rem}

From Theorem \ref{thm:dbrgHT}\revC{}{,} we obtain the following known corollary for distance-regular graphs \revC{\cite{B1993}}{(Biggs
~\cite{B1993})}, which was first shown using ideas from potential theory.

\begin{cor}[\cite{B1993}]
  Let $G$ be a distance-regular graph with intersection array $\{b_0=k,b_1,\ldots,b_{\Delta_G-1};c_1=1,\ldots,c_{\Delta_G}\}$,
  and define
  \begin{align*}
    B&=\begin{pmatrix}
      0 & c_1 &  & &&\\
      b_0 & a_1 &&&&\\
       & b_1 &&&&\\
      &     &&\ddots &c_{\Delta_G-1}&\\
    &&&&a_{\Delta_G-1} &c_{\Delta_G}\\
      &&&&b_{\Delta_G-1}&a_{\Delta_G}
    \end{pmatrix}.
  \end{align*}
  Then,
  \[
    H(G;(v,u))=\bold{H}(T(B)_{-\G_0})_{\G_i},
  \]
  \revA{}{where $\G_i$ denotes the $i$-th row or column of the matrix.}
  \end{cor}

Similarly, from Theorem \ref{thm:dbrgHT}, we can also derive the average hitting time for distance-biregular graphs, which has been recently shown by \revC{Carmona et al. \cite{cej2022}}{Carmona, Encinas and Jiménez~\cite{cej2022}} using potential theory techniques.

%%%%%%%%%%%%%%%%%%%%%%%%%%%%%%%%%%%%%%%%%%%%%%%%%%%%%%%%%%%%%%%%%%%%%%%%%%%%%%%%%%%%%%%%% 
\subsection{The average hitting time of graphs in $t$-distance-regular graphs}\label{sec:as}
%%%%%%%%%%%%%%%%%%%%%%%%%%%%%%%%%%%%%%%%%%%%%%%%%%%%%%%%%%%%%%%%%%%%%%%%%%%%%%%%%%%%%%%%%%

In this section\revC{}{,} we compute the average hitting time in $t$-distance-regular graphs.  
In fact, since $t$-distance-regular graphs are always $f$-equitable, we can compute it using Theorem \ref{thm:equi_old}.

First, we need \revC{}{to} show that the graphs obtained by the corresponding association \revC{scheme}{schemes} are $f$-equitable.

Let $G_1,\dots, G_d$ be a set of graphs with a common vertex set $X$, and assume that $G_i$ and $G_j$ do not have common edges for all $i, j \in \{1, 2, \dots, d\}$. 
Then, for each $v \in X$, we can consider a partition $\pi(v)$ of $X$ as follows:
\[
  \pi(v) = \{N_{G_0}(v), N_{G_1}(v), \dots, N_{G_d}(v)\},
\]
where $N_{G_i}(v)$ represents the neighborhood set of the vertex $v$ in the graph $G_i$.
\revA{Additionally, define the function $f: X \times X \longrightarrow \{0, \dots, d\}$ as $f(v,u) = j$, where $u$ is adjacent to $v$ in $G_j$ and 
define $A_i$ as the adjacency matrix of $G_i$, and $A_0$ as the identity matrix.}
{
Note that $ \pi(v) $ can also be defined by a specific function. 
Define $ f: X \times X \longrightarrow \{0, \dots, d\} $ such that $ f(v,v)=0 $ for all $ v\in X $, and $ f(v,u) = j $ if $ u $ is adjacent to $ v $ in $ G_j $. For a given $ v \in X $, define $ F_i(v)=\{u\in X \mid f(v,u)=i\} $ as a subset of $ X $. 
Then,
\[
  \pi(v)=\{F_{0}(v),F_{1}(v),\ldots,F_{d}(v)\}.
\]
A symmetric association scheme is defined by the adjacency matrices of $ \{G_i\}_{i=0}^d $ satisfying specific conditions. 
Let $ A_i $ be the adjacency matrix of $ G_i $, where $ A_0 $ is the identity matrix.
}
From \revC{\cite{gm1995}}{Godsil and Martin~\cite{gm1995}}, the set $\{A_0, \dots, A_d\}$ forms a symmetric association scheme if and only if, for each $i \in \{1, \dots, d\}$, the partition $\pi(v)$ is equitable for $G_i$, and the quotient matrix of it only depends on $i$. 
Let \revA{$p_{ij}^k = P^k$}{$p_{ij}^k$} be the intersection number of $\{A_0, A_1, \dots, A_d\}$. Then, \revA{the quotient matrix of $G_i$ with respect to the partition $\pi(v)$ is $P^i$.}
{the quotient matrix of $G_k$ with respect to the partition $\pi(v)$, denoted by $P^k$, is given by $P^k=(p_{ij}^k)$.
Note that for simplicity, we index $P^k$ by $\{0,\ldots,d\}$ instead of $\{F_0(v),\ldots,F_d(v)\}$.}

This implies that $G_i$ is $f$-equitable for all $i \in \{1, \dots, d\}$. 
Therefore, from Theorem \ref{thm:equi_old}, we obtain the average hitting times in $G_i$ from vertex $u$ to $v$ as follows:
\[
  H(G_i; (v,u)) = \bold{H}(P^i_{-0})_{f(v,u)}.
\]

The concept of $f$-equitable graph can be extended similarly to an association scheme.

\begin{df}
Let $G$ be a graph with vertex set $V$, and let $A$ be the adjacency matrix of $G$.
Define $f$ as a function $f : V \times V \to \{0, 1, \dots, r\}$, and define $A_i$ as the matrix whose $uv$ entry is $1$ if $f(u,v) = i$.
Then, $G$ is \emph{$f$-equitable} if and only if $A_0, A_1, \dots, A_r$ satisfy the following properties:
\begin{itemize}
  \item $A_0 = I$.
  \item $\sum_{j=0}^r A_j = J$, where $J$ is the all-ones matrix.
  \item For all $i$ satisfying $0 \leq i \leq r$, there exist integers $q_{ij}$ such that 
  \[
    A_i A = \sum_{j=0}^r q_{ij} A_j.
  \] 
\end{itemize}
\end{df}
As in the previous case, $G$ is \emph{weakly $f$-equitable} if and only if $A_0, A_1, \dots, A_r$ satisfy the following properties:
\begin{itemize}
  \item There exists $c$ such that $\sum_{i=0}^{c}A_i=I$.
  \item $\sum_{j=0}^r A_j=J$.
  \item For all $i$ satisfying $0 \leq i \leq r$, there exist integers $q_{ij}$ such that 
  \[
    A_i A = \sum_{j=0}^r q_{ij}A_j.
  \] 
\end{itemize}

Therefore, \revC{$f$-equitable}{$f$-equitablity} is a generalization of an association scheme, and weakly \revC{$f$-equitable}{$f$-equitablity} is a generalization of a coherent configuration.

We will use the above concepts to obtain the average hitting time of  $t$-distance-regular graphs, which we define next.

  Let $G$ be a graph with edge set $E$, and let $E_1,\ldots, E_t$ be a partition of $E$.  
Let $\xi=(v_1,\ldots, v_L)$ be a walk on $G$, and define the $t$-\emph{length} of $\xi$, denoted as $\revC{l_t(\xi)}{\boldsymbol{l_t(\xi)}}$, as the vector in $\mathbb{N}^t$ as follows:  
\revA{\[
  l_t(\xi)=(|\{j\mid v_j\in E_1\}|,|\{j\mid v_j\in E_2\}|,\ldots,|\{j\mid v_j\in E_t\}|).
\]  }
{
\[
  \boldsymbol{l_t(\xi)}=(|\{j\in\{1,\ldots,L\}\mid v_j\in E_1\}|,\ldots,|\{j\in\{1,\ldots,L\}\mid v_j\in E_t\}|).
\]  
}
Let $\leq$ be a monomial order (an example of a monomial order is the lexicographical order) on $\mathbb{N}^t$.  
The $t$-\emph{distance} $d_t$ between two vertices $u$ and $v$ of $G$ is defined as  
\[
  d_t(v,u)=\min_{\leq}\{\revC{l_t(\xi)}{\boldsymbol{l_t(\xi)}}\mid \text{$\xi$ is a walk between $u$ and $v$}\}.
\]  
Then, a $t$-distance-regular is defined as \revC{follows}{follow}:
 
\begin{df}[$t$-distance-regular]
A graph $G$ is $t$-\emph{distance-regular} if there exists an edge partition $\{E_i\}_{i=1}^t$ and a monomial order $\leq$ on $\mathbb{N}^t$ such that for any $\bold{a},\bold{b} \in \mathbb{N}^t$, the number of vertices that are at $t$-distance $\bold{a}$ from $u$ and $t$-distance $\bold{b}$ from $v$ depends only on the $t$-distance between $u$ and $v$.      
\end{df}

It is known that \revC{an}{a} $t$-distance-regular graph and a multivariate $P$-polynomial association scheme are in one-to-one correspondence, see \revC{\cite{BVZZ2023}}{Bernard, Crampe, Vinet, Zaimi, and Zhang
~\cite{BVZZ2023}}.

Let $G$ be a $t$-distance-regular graph with vertex set $V$ and adjacency matrix $A$.
Then, there exists a symmetric association scheme $\mathfrak{X}=(V,\{A_i\}_{i=0}^d)$ such that there exist $e_1,e_2,\dots,e_t\in \{0,1,\dots,d\}$ such that 
\[
  A=\sum_{i=1}^{t}A_{e_i}.
\]
Therefore, $G$ is also $f$-equitable with quotient matrix 
\[
  q_{ij}=\sum_{x=1}^{t}p_{ij}^{e_x},
\]
where $p_{ij}^{e_x}$ are the intersection numbers of $\mathfrak{X}$.
Hence, using Theorem \ref{thm:eqRw}\revC{}{,} we can obtain the average hitting times of a $t$-distance-regular graph as follows.

\begin{cor}
  Let $G$ be a $t$-distance-regular graph with adjacency matrix $A$. Let $\mathfrak{X}=(V,\{A_i\}_{i=0}^d)$ be a symmetric association scheme corresponding to the $t$-distance matrices of $G$,
  and let $p_{ij}^k$ be the intersection numbers of $\mathfrak{X}$.
  Define the matrix $B$ \revA{}{indexed by $\{0,\ldots,d\}$,} with entries given by
  \[
    q_{ij}=\sum_{x=1}^{t} p_{ij}^{e_x},
  \]
  where $e_x$ satisfies 
  \[
    A=\sum_{x=1}^{t}A_{e_x}.
  \]
  Then, for any vertices $u$ and $v$ in $G$,
  \[
    H(G; (v,u)) = \mathbf{H}(B_{-0})_{f(v,u)}.
  \]  
\end{cor}

%%%%%%%%%%%%%%%%%%%%%%%%%%%%%%%%%%%%%%%%%%%%%%%%%%%%%%%%%%%%%%%%%%%%%%%%%%%%%%%%%%%%%%%%%%
\subsection{The average hitting time of cone graphs}\label{sec:cone}
%%%%%%%%%%%%%%%%%%%%%%%%%%%%%%%%%%%%%%%%%%%%%%%%%%%%%%%%%%%%%%%%%%%%%%%%%%%%%%%%%%%%%%%%%%

In this section, we prove the average hitting time of cone graphs in both a simple random walk and a maximal entropy random walk.

Unlike \revC{}{what} we saw for some generalizations of distance-regular graphs, a cone graph is generally neither regular nor biregular. Therefore, the average hitting times by the two types of random walks (simple random walk and maximal entropy random walk) differ.

\begin{df}\label{df:cone}
  The \emph{cone graph} over a graph $G$, denoted as $G^+$, is the graph whose vertex set consists of the vertices of $G$ and an additional vertex $\revC{\infty}{a}$. 
  Its edge set includes all edges of $G$ as well as edges of the form $\{\revC{\infty}{a},v\}$, where $v$ is any vertex in $G$. 
\end{df}

If $G$ has certain regularity properties, then we can compute the average hitting times in $G^+$. For instance, if $G$ is $k$-regular, then $G^+$ is a pseudo-distance-regular graph around a vertex, and we can use Theorem \ref{thm:stabHtW} to obtain the average hitting time from a vertex to $\revC{\infty}{a}$. Also, if $G$ is $f$-equitable, then $G^+$ is weakly-$f$-equitable, and we can use Theorem \ref{thm:equiW} to show the average hitting times in $G^+$.

First, we consider the case when $G$ is $k$-regular. In this case, $G^+$ is pseudo-distance-regular around the vertex $\revC{\infty}{a}$, because $P_{\revC{\infty}{a}}=\{\G_{0}(\revC{\infty}{a}),\G_1(\revC{\infty}{a})\}=\{\{\revC{\infty}{a}\},V(G)\}$ is an equitable partition, and using our framework\revC{}{,} we can obtain the following result.

\begin{thm}\label{thm:conePHT}
  Let $G$ be a $k$-regular graph with order $n$, and let $G^+$ be a cone graph over $G$.
  Then,
  \begin{align*}
    H(G^+;(\revC{\infty}{a},v))&=k+1,\\
    H_m(G^+;(\revC{\infty}{a},v))&=\frac{\lambda_1^2}{n},
  \end{align*}
  where $v$ is any vertex in $G$ and $\lambda_1=\frac{k}{2}+\sqrt{n+\frac{k^2}{4}}$.
\end{thm}

\begin{proof}
  
  Let $A$ be the adjacency matrix of $G$.
  Then, the adjacency matrix of $G^+$ is
  \[
    A^+=\begin{pmatrix}
      0 &\allone\\
      \allone^\top &A
    \end{pmatrix}.  
  \]
  Therefore, the maximum eigenvalue of $A^+$ is $\lambda_1$, and 
  \[
    \revC{\nu}{\boldsymbol{\nu}}=\revC{(\sqrt{n+\frac{k^2}{4}}-\frac{k}{2},1,\ldots,1)}{\left(\sqrt{n+\frac{k^2}{4}}-\frac{k}{2},1,\ldots,1\right)}=(\lambda_1-k,1,\ldots,1)
  \]
  is \revC{}{the} eigenvector of $A^+$ corresponding to $\lambda_1$.
  Additionally, $G^+$ is pseudo-distance-regular and the quotient matrix of $P_{\revC{\infty}{a}}=\{\{\revC{\infty}{a}\},V(G)\}$ is
  $B(\revC{\infty}{a})=\begin{pmatrix}
    0 & 1\\
    n&k
  \end{pmatrix}.$
  Then, it follows that 
  \begin{align*}
    D_\nu B(\revC{\infty}{a})D_\nu^{-1}&=\begin{pmatrix}
      \lambda_1-k &0\\
      0&1
    \end{pmatrix}\begin{pmatrix}
      0 &1\\
      n&k
    \end{pmatrix}\begin{pmatrix}
      \frac{1}{\lambda_1-k} &0\\
      0&1
    \end{pmatrix}\\
    &=\begin{pmatrix}
      0 &\lambda_1-k\\
      \lambda_1&k
    \end{pmatrix}.
  \end{align*}
  Finally, using Theorem \ref{thm:pseudo}\revC{}{,} we obtain
  \begin{align*}
    H(G^+;(\revC{\infty}{a},v))&=-(\frac{k}{k+1}-1)^{-1}\\
    &=k+1,\\
    H_m(G^+;(\revC{\infty}{a},v))&=-\frac{k-\lambda_1}{\lambda_1}\\
    &=\frac{\lambda_1^2}{n}.\qedhere
  \end{align*}

\end{proof}
 
 Rao \cite{R2013} also investigated hitting times of graphs with symmetry by using graph \revC{autormophisms}{automorphisms}. However, we observed that the condition for the automorphism in the original Theorem  2.1 is incorrect. We state the correct statement below, since this will be used later on.

    \begin{thm}[\cite{R2013}, Theorem 2.1]\label{thm:R2013}
    Let $G$ be a connected graph, let $v$ be a vertex of $G$, and let $u$ be a neighbor of $v$.  
    If, for any vertex $x\in G(v)$, there exists an automorphism $\sigma$ of $G$ such that $\sigma(v)=v$ and $\sigma(x)=u$, then the average hitting time from $u$ to $v$ is $\frac{2e}{k}-1$, where $e$ is the number of edges in the graph and $k$ is the degree of $v$.
    \end{thm}
    When $G$ is vertex-transitive, we can apply Theorem \ref{thm:R2013} and obtain $H(G^+;(\revC{\infty}{a},v))=d+1$.
    Actually, Theorem \ref{thm:R2013} can be generalized using a stabilized equitable partition (see Definition \ref{def:stabilizedequitablepartition}) as follows.
    
    \begin{thm}\label{thm:genR}
        Let $G$ be a connected graph, let $v$ be a vertex of $G$, and let $u$ be a neighbor of $v$.  
        If there exists a stabilized equitable partition $P$ centered at $v$ such that  
    \[
        P=\{V_0=\{v\},V_1=G(v),V_2,\dots,V_r\},
    \]  
        then the average hitting time from $u$ to $v$ is $\frac{2e}{k}-1$, where $e$ is the number of edges in the graph and $k$ is the degree of $v$.  
    \end{thm}
    \begin{proof}
        From Theorem \ref{thm:stabHt}, for every vertex $x\in G(v)$, we obtain $H(G;(v,u))=H(G;(v,x))$.  
        Then, using the same proof argument as in Theorem \ref{thm:R2013}, the result is obtained.
    \end{proof}
    Note that when the conditions on $u$ and $v$ in Theorem \ref{thm:R2013} hold, there exists a stabilized equitable partition $P$ that satisfies the conditions of Theorem \ref{thm:R2013}. 
    We can also calculate $H(G^+;(\revC{\infty}{a},v))$ using Theorem \ref{thm:genR} even if $G$ is not vertex-transitive.

    The next result shows that we can also apply Theorem \ref{thm:genR} to graphs that are obtained from an association scheme.
    \begin{cor}\label{cor:asHt}
        Let $\mathfrak{X} = (V, \{A_i\}_{i=0}^d)$ be a symmetric association scheme, and let $G$ be a graph with adjacency matrix $A_i$, where $1 \leq i \leq d$.  
        If $G$ is connected, then for any adjacent vertices $u$ and $v$ in $G$,  
        $H(G; (v, u)) = |V| - 1$.  
    \end{cor}
    \begin{proof}
         This result is obtained using Theorem \ref{thm:genR} and the fact that, for any vertex $v$ in $G$, there exists a partition $\pi(v)$ such that 
    \[
        \pi(v)=\{V_0=\{v\},V_1=G(v),V_2,\dots,V_r\},
    \]
    where the construction of $\pi(v)$ is the same as in Section \ref{sec:as}.
    \end{proof}

Next, we consider the case when $G$ is $f$-equitable. If $G$ is $f$-equitable, $G^+$ has an even stronger property.

\begin{propo}\label{thm:cone}
  Let $G$ be a $k$-regular and $f$-equitable graph with quotient matrix $B$.
  Then, the cone graph over $G$ is weakly $f^+$-equitable, where $f^+$ is the function 
  \[
    f^+(v,u)=\begin{cases}
      f(v,u) & u\neq \revC{\infty}{a}\land v\neq \revC{\infty}{a},\\
      \revC{\infty}{a}_0 & u=v=\revC{\infty}{a},\\
      \revC{\infty}{a}_1 & \text{otherwise},
    \end{cases}
  \] 
  \revC{}{where $a_0$ and $a_1$ are some values not in the image of $f$.}
  \revC{and}{Additionally,} it has the following quotient matrices  
  \begin{align*}
    B(x_0)&=\left(\begin{array}{ccc|c}
      & & & k_{x_0}\\
      & {\Huge B} & & \vdots\\
      &&&k_{x_r}\\
      \hline 
      1 & \ldots & 1 &0
      \end{array}\right)
    ,\\
    B(\revC{\infty}{a}_0)&=\begin{pmatrix}
      0 & 1\\
      n&k
    \end{pmatrix},    
  \end{align*}
  where $k_{x_i}$ is the size of $F_{x_i}(v)$ for any vertex $v\in G$.
\end{propo}

\begin{proof}
  Let $v$ be any vertex in $G$\revA{.}{, and define $F^+_{x_i}(v)\coloneqq\{u\in V\mid f^+(v,u)=x_i\}$.}
  We consider 
  \[
    P^+(v)=\{F^+_{x_0}(v),F^+_{x_1}(v),\ldots,F^+_{x_r}(v),F^+_{\revC{\infty}{a}_1}(v)\}.
  \]
  From the definition of $f^+$, it follows that $F^+_{x_i}(v)=F_{x_i}(v)$ and $F^+_{\revC{\infty}{a}_1}(v)=\{\revC{\infty}{a}\}$.
  Therefore,
  \[
      |G(u)\cap F^+_{x_i}(v)|=\begin{cases}
        q_{i,f(v,u)} & u\neq\revC{\infty}{a},\\
        k_{x_i} & u=\revC{\infty}{a}.
      \end{cases}
  \]
  This implies that $P^+(v)$ is equitable, and we can also obtain the entries of the quotient matrix.
 
  Since $G$ is $k$-regular, the partition $P^+(\revC{\infty}{a})=\{F^+_{\revC{\infty}{a}_0}(\revC{\infty}{a})=\{\revC{\infty}{a}\},F^+_{\revC{\infty}{a}_1}(\revC{\infty}{a})=V\}$ is equitable, and the quotient matrix of $P^+(\revC{\infty}{a})$ is $B(\revC{\infty}{a}_0)$.
  Thus, the proof is completed.
\end{proof}

Using Proposition \ref{thm:cone}\revC{}{,} we can show the main result for cone graphs.

\begin{thm}\label{thm:coneHT}
  Let $G$ be a $k$-regular and $f$-equitable graph with quotient matrix $B$, and order $n$. 
  Define $\lambda_1=\frac{k}{2}+\sqrt{n+\frac{k^2}{4}}$ and 
  $B^+$ as follows
  \[
    B^+=\left(\begin{array}{ccc|c}
      & & & k_{x_0}\\
      & {\Huge B} & & \vdots\\
      &&&k_{x_r}\\
      \hline 
      1 & \ldots & 1 &0
      \end{array}\right).    
  \]
  Let $(B^*)^{\top}=D_\nu B^+D_\nu^{-1}$, $h_{x}=\bold{H}(T(B^+)_{-x_0})_x$ and $h'_x=\bold{H}(\frac{1}{\lambda_1}(B^*)^{\top}_{-x_0})_x$.
  Then,
  \begin{align*}
    H(G^+;(v,u))&=\begin{cases*}
      k+1 &$v=\revC{\infty}{a}$,\\
      h_{f^+(v,u)} & \text{otherwise},
    \end{cases*}\\
    H_m(G^+;(v,u))&=\begin{cases*}
      \frac{\lambda_1^2}{n}&$v=\revC{\infty}{a}$,\\
      h'_{f^+(v,u)} & \text{otherwise.}
    \end{cases*}
  \end{align*}
\end{thm}

\begin{proof}
  Since $G$ is $k$-regular, using Theorem \ref{thm:conePHT}, we obtain 
  \begin{align*}
    H(G^+;(\revC{\infty}{a},u))&=k+1,\\
    H_m(G^+;(\revC{\infty}{a},u))&=\frac{\lambda_1^2}{n}.
  \end{align*}

  From Proposition \ref{thm:cone}\revC{}{,} we know that $G^+$ is weakly-$f^+$-equitable.
  Therefore, when $v\neq\revC{\infty}{a}$, we apply Theorem \ref{thm:equi} and obtain
  \begin{align*}
    H(G;(v,u))&=h_{f^+(v,u)},\\
    H_m(G;(v,u))&=h'_{f^+(v,u)}.\qedhere
  \end{align*} 
\end{proof}

When $G$ is a cycle graph (distance-regular), then $G^+$ is a wheel graph. Therefore, we can use Theorem \ref{thm:coneHT} to obtain the average hitting times in wheel graphs, providing an alternative proof and an extension of the results by Yang \cite{Y2011}.

%%%%%%%%%%%%%%%%%%%%%%%%%%%%%%%%%%%%%%%%%%%%%%%%%%%%%%%%%%%%%%%%%%%%%%%%%%%%%
\subsection*{Acknowledgements}
%%%%%%%%%%%%%%%%%%%%%%%%%%%%%%%%%%%%%%%%%%%%%%%%%%%%%%%%%%%%%%%%%%%%%%%%%%%%%
Aida Abiad is supported by the Dutch Research Council (NWO) through the grant VI.Vidi.213.085.  

%%%%%%%%%%%%%%%%%%%%%%%%%%%%%%%%%%%%%%%%%%%%%%%%%%%%%%%%%%%

\end{document}